\documentclass[12pt,leqno,twoside]{amsart}
\usepackage{amssymb,amsfonts,amsmath,amsthm, amstext}
\usepackage{latexsym}
\usepackage{verbatim}
\usepackage{graphicx}

\usepackage{pgf,tikz}
\usepackage{float}

\usepackage[latin1]{inputenc}
\usepackage[T1]{fontenc}
\usepackage[colorlinks=true, pdfstartview=FitV, linkcolor=blue, citecolor=blue, urlcolor=blue]{hyperref}
\usepackage{amstext,amsmath,amscd, bezier,indentfirst,amsthm,amsgen,enumerate}
\usepackage[all,knot,arc,import,poly]{xy}
\usepackage{amsfonts,xcolor, soul}  
\usepackage{graphicx}
\usepackage{srcltx}
\usepackage{enumitem}
\usepackage{epsfig}

\usepackage{tikz}
\usepackage{mathrsfs}
\usepackage[margin=1in]{geometry}
\usepackage[normalem]{ulem}

\newtheorem{theorem}{Theorem}[section]
\newtheorem{corollary}[theorem]{Corollary}
\newtheorem{lemma}[theorem]{Lemma}

\newtheorem{conjecture}[theorem]{Conjecture}
\theoremstyle{definition}
\newtheorem{definition}[theorem]{\sc Definition}
\theoremstyle{remark}
\newtheorem{remark}[theorem]{\sc Remark}
\theoremstyle{remark}

\theoremstyle{remark}
\newtheorem{example}[theorem]{\sc Example}
\theoremstyle{remark}

\theoremstyle{remark}

\newtheorem{lem/defi}[thm]{Lemma/Definition}
\newtheorem{defi/lem}[thm]{Definition/Lemma}
\newtheorem{prop/defi}[thm]{Proposition/Definition}

\theoremstyle{definition}

\theoremstyle{remark}

\renewcommand{\Box}{\square}    

\newcommand{\gr}{{\mathrm{gr}}}

\newcommand{\hot}{\mathrm{h.o.t.}}

\newcommand{\Sing}{{\mathrm{Sing}}}

\newcommand{\Lip}{{\mathrm{Lip}}}
\renewcommand{\top}{{\mathrm{top}}}
\newcommand{\an}{{\mathrm{an}}}

\newcommand{\im}{\mathop{\rm{Im}}\nolimits}
\newcommand{\mult}{{\rm{mult}}}

\newcommand{\ord}{\mathop{\rm{ord}}}

\newcommand{\cone}{\mathop{\rm{Cone }}\nolimits}

\newcommand{\grad}{\mathop{\rm{grad}}\nolimits}
\newcommand{\Grad}{\mathop{\rm{grad}}\nolimits}

\newcommand{\Horn}{{\rm{Horn}}}
\newcommand{\e}{\varepsilon}
\newcommand{\mi}{\setminus}
\newcommand{\fin}{\hspace*{\fill}$\Box$}

\newcommand{\dist}{\operatorname{dist}}

\newcommand{\PC}{\mathcal{PC}}
\newcommand{\GC}{\mathcal{GC}}
\newcommand{\mcd}{\mathcal{D}}

\newcommand{\be}{\begin{equation}}
\newcommand{\ee}{\end{equation}}

\newcommand{\beqn}{\begin{eqnarray*}}
\newcommand{\eeqn}{\end{eqnarray*}}

\newcommand{\cC}{{\mathcal C}}

\newcommand{\PP}{\mathcal{NP}}

\newcommand{\bC}{{\mathbb C}}

\newcommand{\F}{\mathbb{F}}

\newcommand{\bN}{{\mathbb N}}

\newcommand{\bQ}{{\mathbb Q}}

\begin{document}

\title[Clustering polar curves]{Clustering polar curves}

\author{Piotr Migus}
\address{Air Force Institute of Technology,
ul. Ksi\c{e}cia Boles\l awa 6,
01-494 Warsaw, Poland}
\email{migus.piotr@gmail.com}

\author{Lauren\c tiu P\u aunescu}
\address{School of Mathematics and Statistics, University of Sydney,
  Sydney, NSW, 2006, Australia.}
\email{laurent@maths.usyd.edu.au}

\author{Mihai Tib\u ar}
\address{Math\' ematiques, UMR 8524 CNRS,
Universit\'e de Lille, \  59655 Villeneuve d'Ascq, France.}
\email{mtibar@univ-lille.fr}

\dedicatory{In memoriam \c Stefan Papadima}

\thanks{P.M. and M.T. acknowledge the support of the Labex CEMPI
(ANR-11-LABX-0007-01). P.M. was also partially supported by NCN-Poland, grant number 2018/02/X/ST1/02699.}

\begin{abstract}
This essay builds on the idea of grouping  the polar curves of  2-variable function germs in polar clusters. In the topological category, one obtains a bijective correspondence between  certain partitions of the polar quotients of two topologically  equivalent  function germs.  We explain how this bijective correspondence may be refined in the Lipschitz category in terms of the associated gradient canyons.
 \end{abstract}

\maketitle
\setcounter{section}{0}


\section{Introduction}

The classical results  \cite{Zar, Zar2}  by Zariski  and  \cite{Bur} by Burau  tell that if two holomorphic function germs $f,g:(\bC^2,0)\rightarrow(\bC,0)$ are topological right-equivalent,  
then this topological  equivalence induces a bijective correspondence between the irreducible components of the curves $\{f=0\}$ and $\{g=0\}$ such that their Puiseux pairs are preserved, as well as the intersection multiplicities of each pair of components. The reciprocal was proved by Parusi\'nski \cite{Pa}. The  polar curve of a function germ encodes information on its topology and beyond. However, the polar curves are more subtle to decode since, unlike for the curves $\{f=0\}$ and $\{g=0\}$, there is no bijective correspondence between the irreducible components of the corresponding polars, even if $f$ and $g$ are analytically right-equivalent. The topological type of a curve does not characterise the topological type of its generic polar curve, see \cite{Me}, \cite{kuo-lu}. Teissier \cite{Te} proved that the polar quotients are topological invariants (see Theorem \ref{teissier}) and Kuo and Lu \cite{kuo-lu} studied the contact of the polar branches via the Newton polygon. Many newer  results enrich periodically the landscape of polar curves in relation to the topological, to the analytical, and to the  Lipschitz structures of singularities.

We remark that tracking the contact orders of the polar arcs and of the  roots of $f$  as  in \cite{kuo-lu} induces a natural partition of the set of polar arcs into  clusters, in such a way that the classical bijective correspondence of branches of topologically right-equivalent  function germs $f$ and $g$ induces a bijective correspondence of those clusters (Theorem \ref{polar_cluster2}). 

We point out several consequences of this clustering. 
To each polar cluster we associate a \L ojasiewicz exponent which is a topological invariant (Corollaries 
 \ref{c:lojacluster} and  \ref{c:lojacluster2}).  

%
 
 Similarly, for each  line of the  tangent cone of the zero set of $f$ at the origin, the polar clusters which are tangent to it yield a tangential Milnor number which is a topological invariant (Corollary \ref{c:lojacluster3}), and  their sum  equals the Milnor number of $f$. 
 
 Passing to the Lipschitz category (in the sens of the right bi-Lipschitz equivalence of function germs), we explain how this clustering evolves into a more refined class of objects called  \emph{gradient canyons}, cf \cite{KKP}, which are arc-neighbourhoods of certain topological sub-clusters. 
  The set of gradient canyons may then be partitioned into certain clusters that turn out to define discrete bi-Lipschitz invariants. Following \cite{PT}, we survey the bi-Lipschitz invariance of the gradient canyons, and we show how one may deduce from it the Henry-Parusi\'nski continuous bi-Lipschitz moduli \cite{HP}, \cite{HP2}.   
  
  \smallskip
  
This essay offers a systematic point of view on a couple of classical equivalence relations between two-variables holomorphic function germs,  embedding  the latest developments and thus helping the interested reader to easier access the current research status, and  giving  a taste of open questions in higher dimensions.

\medskip

\noindent{\bf Acknowledgements.} We thank the anonymous  referee for the careful reading of the manuscript and for the interesting and helpful remarks
which helped us to improve the exposition. 

\section{Polar quotients and \L ojasiewicz inequalities}

\subsection{Polar quotients}

Let $f:(\bC^n,0)\rightarrow(\bC,0)$ be a holomorphic function germ, and let $m:= \ord_{0}f$.  
Let $l=\sum_{i=1}^n a_i x_i$ be a non-zero  generic linear form. The \emph{polar curve of $f$ relative to $l$} is the germ at 0  of the analytic closure of $\Sing (l,f) \mi \Sing f$, denoted by $\Gamma(l,f)$. If $l$ if \emph{generic}, then either $\dim \Gamma(l,f) =1$
or $\Gamma(l,f)=\emptyset$.
Let $\Gamma_1, \ldots ,\Gamma_{s}$ be the irreducible components of $\Gamma(l,f)$, with multiplicities $m_{j}:= \mult_{0}(\Gamma_j)$. 

The \emph{polar quotients of $f $ relative to $l $} are the rational numbers:
\begin{equation}\label{eq:quot}
 q_{j} :=\mult_{0}(\{f = 0\},\Gamma_j) / m_{j}
\end{equation}
as introduced by Teissier \cite{Te}. They are related to the Milnor number $\mu(f)$ by the equality:
$$ \mu(f) = \sum_{j=1}^{s}\mult_{0}(\{f = 0\},\Gamma_{j}) - m_j = \sum_{j=1}^{s} m_{j}(q_{j}-1).$$


Teissier \cite{Te1} 
showed that this set of polar quotients  $Q(f, l)$ does not depend on the generic linear $l$, therefore one may use the notation $Q(f) := Q(f, l)$. 

\subsection{Puiseux expansions}
One has the following very useful interpretation of polar quotients in terms of Puiseux expansions.
Consider a holomorphic map germ 
$$
\alpha :(\bC,0)\to(\bC^2,0), \quad \alpha(t)=(z(t),w(t)) \not \equiv 0.
$$
The image set germ $\alpha_*=\im(\alpha)$ is a \emph{curve germ} at $0\in \bC^2$, also called a \emph{holomoprhic arc} at $0$. Then there is a well-defined tangent line $T(\alpha_*)$ at $0$ and $T(\alpha_*)\in\bC P^1$. 

Let $\F$ be the field of convergent fractional power series in an indeterminate $y$. By the Newton-Puiseux Theorem we have that $\F$ is algebraically closed (\cite{BK}, \cite{Wa}).

A non-zero element of $\F$ has the form 
\begin{equation}\label{puis}
\eta(y)=a_0y^{n_0/N}+a_1y^{n_1/N}+a_2y^{n_2/N}+\cdots,\quad n_0<n_1<n_2<\cdots,
\end{equation}
where $a_i \in \bC^*$ and  $N,n_i\in \bN$ with $\gcd(N,n_0,n_1,\dots)=1$, $\lim \sup |a_i|^{\frac{1}{n_i}}<\infty$. The elements of $\F$ are called \emph{Puiseux arcs}. There are $N-1$ \emph{conjugates} of $\eta,$ which are the Puiseux arcs of the form
$$
\eta^{(k)}_{conj}(y):=\sum a_i\varepsilon^{kn_i}y^{n_i/N}, \quad \varepsilon:=e^{\frac{2\pi\sqrt{-1}}{N}},
$$
where $k\in\{0,\dots ,N-1\}$.

By the \emph{order of a Puiseux arc} \eqref{puis} we mean $\ord \eta(y):=\frac{n_0}{N}$ and by the \emph{Puiseux multiplicity} we mean $m_{puiseux}(\eta)=N$ (\cite{BK}, \cite{Wa}). Let us also set the notation $\F_1:=\{\eta \in \F:\ord \eta (y)\geq 1\}$.

For any $\eta \in \F_1$, $\ord \eta (y)\geq 1$, the following map germ:
$$
\eta_{par} :(\bC,0)\to(\bC^2,0), \quad t\mapsto (\eta(t^N),t^N), \quad N:=m_{puiseux}(\eta),
$$
is holomorphic, and all the conjugates of $\eta$ lead to the same irreducible curve $\im\eta_{par}$, which will be denoted by $\eta_*$.

\begin{definition}[Contact order, cf e.g. \cite{BK}, \cite{Wa} or \cite{KKP}]\ \\
 The \emph{contact order} between two different holomorphic arcs $\alpha_*$ and $\beta_*, \alpha, \beta \in \F_1$ is defined as: 
 $$
 \max \ord_{y} (\alpha'(y) - \beta'(y))
 $$
where the maximum is taken over all conjugates $\alpha'$ of $\alpha$ and $\beta'$ of $\beta$.
\end{definition}

%
\medskip

Let $f:(\bC^n,0)\rightarrow(\bC,0)$ be a holomorphic function germ, and let $m:= \ord_{0}f$. We say that $f$ is \emph{mini-regular in $x_1$ of order $m$},  if the initial form of the Taylor expansion of $f$ is not equal to 0 at the point $(1,0, \dots,0)$, in other words $f_m(1,0,\dots,0)\neq 0$ where $f(x_1,\dots, x_n)=f_{m}(x_1,\dots, x_n)+f_{m+1}(x_1,\dots, x_n)+\hot$ is the homogeneous Taylor expansion of $f$.

Let $f:(\bC^2,0)\to (\bC,0)$ be a mini-regular holomorphic function.
 We have the following Puiseux factorisations of $f$ and $f_x$
$$
f(x,y)=u\cdot \prod_{i=1}^m(x-\zeta_i(y)), \quad f_x(x,y)=v\cdot \prod_{i=1}^{m-1}(x-\gamma_i(y)),
$$ 
where $u,v$ are units, in the sense that the functions $u$ and $v$ do not vanish at the origin. \\ 
 If $f_x=g_1^{q_1} \cdots g_p^{q_p}$ is the decomposition into irreducible factors, then $\Gamma_i=\{g_i=0\}, i=1,\dots, p,$ if $g_i$ is not a factor of $f$ as well;  there exists at least one Puiseux root $\gamma$ of $f_x$ such that $g_i(\gamma(y),y)\equiv 0$. 
 Note that in the definition of the set of polar quotients $Q(f)$ we do not consider Puiseux arcs $\gamma$ such that $f(\gamma(y),y)\equiv 0$ and $f_x(\gamma(y),y) \equiv 0$, since these arcs are not components of the polar set, by its definition. However,  sometimes one refers to such arcs also as ``polar arcs'' because they are solutions of the equation $f_x(\gamma(y),y) \equiv 0$.

 The set of the Puiseux roots of $f_x$ we will denote by $P(f)$ and its elements will be called \emph{polar arcs}. The set of the Puiseux roots of $f$  will be denoted  by $Z(f)$.


If $y\mapsto (\gamma(y), y)$ is any  Puiseux parametrisation of $\Gamma_i$, then  the polar quotient $q_{i}$ defined at \eqref{eq:quot} is:
$$q_i = \ord f(\gamma(y),y).
$$


\medskip

\subsection{\L ojasiewicz exponents}
  For an analytic map $G:(\bC^n,0)\to (\bC^m,0)$, one has the following inequality, called \emph{\L ojasiewicz distance inequality}:  
\begin{equation}\label{L1}
 \| G(x)\| \geq C\dist(x,G^{-1}(0))^{\rho}, 
\end{equation}
 which holds in some neighbourhood of $0\in\bC^n$ and for some constant $C>0$.
  The infimum of all such exponents $\rho$ is called \emph{{\L}ojasiewicz exponent} and is denoted by $l_0(G)$. By \cite{LT} we have $l_0(G)$ is rational. For $G = \grad f$, the gradient of an analytic function germ $f:(\bC^n,0)\to (\bC,0)$, we have the {\L}ojasiewicz exponent $l_0(\grad f)$. In the particular case when $f$ has an isolated singularity, the inequality \eqref{L1} for $\grad f$ takes the form 
$$
\|\grad f(x)\|\geq C\|x\|^{\rho}.
$$

Let $f:(\bC^n,0)\rightarrow(\bC,0)$ be an analytic function. It has been proved  by {\L}ojasiewicz \cite{Lo} that there exist a neighbourhood of $0\in \bC^n$ and constants $C,\rho>0$ such that the following 
\emph{{\L}ojasiewicz gradient inequality} holds
\begin{equation}\label{L3}
\|\grad f(x)\|\geq C|f(x)|^{\rho}, \quad \textrm{ for }x\in U.
\end{equation}
The smallest exponent $\rho$ in \eqref{L3} is called \emph{{\L}ojasiewicz exponent of $f$} and will be denoted by $\rho_0(f)$. 

\begin{remark}
 {\L}ojasiewicz, in \cite{Lo}, proved those inequalities for real analytic functions. The complex case may be derived from the real one. Moreover, {\L}ojasiewicz also proved that $\rho_0(f)\in ]0,1[$, see also  \cite{Te}. The fact that $\rho_0(f)$ is rational was proven in \cite{bori}.
\end{remark}

 Let $f,g:(\bC^n,0)\rightarrow(\bC,0)$ be analytic functions. We say that $f$ and $g$ are topologically (analytically or Lipschitz)  equivalent, and denote this  by $f\stackrel{\top}{\sim}g$
 (by $f\stackrel{\an}{\sim}g$ or by $f\stackrel{\Lip}{\sim}g$, respectively), if $g=f\circ \varphi$ for some germ $\varphi: (\bC^n,0)\rightarrow(\bC^n,0)$ of a  homeomorphism (of an analytic isomorphism, or of a bi-Lipschitz homeomorphism, respectively). 

\begin{theorem}\label{teissier} \cite[Corollary 2]{Te} \
Let $f:(\bC^n,0)\rightarrow(\bC,0)$ be a holomorphic function with \emph{isolated singularity}. One has:
\begin{enumerate}
\rm \item \it If $f \stackrel{\an}{\sim}g$  then $Q(f)=Q(g)$.
\rm \item \it $l_0(\grad f)=\max Q(f)-1$ and $\rho_0(f)=(\max Q(f) -1)/ \max Q(f)$. 
\rm \item \it  The \L ojasiewicz exponents $l_0(\grad f)$ and $\rho_0(f)$ are realised along at least one of the polar branches\footnote{Let us point out that the generic polar locus of a function $f$ with true isolated singularity is a non-empty curve.} 
of $f$.
\rm \item \it  In case $n=2$,  if $f\stackrel{\top}{\sim}g$  then $Q(f)=Q(g)$. 
In particular the \L ojasiewicz exponents of $f$ and $g$ are equal. 
\end{enumerate}
\end{theorem}

\begin{proof}[Idea of proof]
The only point which is perhaps not explicitly stated in  \cite{Te} is (c), but it can be found in the proof of Theorem 2, where Teissier shows that the polar pairs are the same as the multiplicity pairs for the integral closure
of the Jacobian ideal. This means that there is some polar arc which minimises the multiplicity pairs.  The polar pairs are finitely many  (positive) integers.
Then \cite[Corollary 2]{Te} just gives the consequences for the exponents: they are realised along some polar arc.
In particular $\rho_0(f)\in ]0,1[ $ and  it is a rational number.
 Teissier uses the normalised blow-up of the conormal space of $f$ and the fact that the Jacobian ideal has only one zero point.
\end{proof}


One may state the following: 

\begin{conjecture}
 Teissier's results (a--d)  hold for $f$ with \emph{non-isolated singularities}\footnote{In case of point (c) we have to add the condition that the generic polar locus is \emph{non-empty} (and hence it is a curve). }, and for any $n\ge 2$.
 \end{conjecture}
 
 The case $n=2$ with non-isolated singularities can be extracted from \cite{kuo-lu} or  \cite{Pa}.  An explicit proof was given more recently in \cite{vietnam}.
 This will also be a consequence of our Theorem \ref{polar_cluster2} on the topological invariance of polar clusters.
 
The statement
\emph{``If $f$ and $g$ are topologically equivalent then the \L ojasiewicz exponents are the same''}  is a  weaker version of the point (d),  called ``Teissier conjecture'' by  Brzostowski \cite{Brz},  earlier known as ``Teissier question''.

 Concerning the above Teissier conjecture some partial results are available in the case  $n> 2$: for  isolated weighted homogeneous singularities in $n=3$ variables \cite[Corollary 2]{KOP}, and  for isolated semi-quasihomogeneous singularities in $n\ge 2$ variables and their deformations  \cite[Corollary 2]{Brz}.


\section{Topological invariance of polar clusters}

 From now on  we will only consider the case of two variables.

Let $\cone_0 \{f=0\} = \bigcup_{k=1}^{r}L_{k}$ be the decomposition of the tangent cone into  tangent lines intersecting at $0$ and let 
$$
P^T(f)=\{\gamma\in P(f)   \mid \gamma_* \textrm{ is tangent to }L_k \textrm{ for some }k\in\{1,\dots,r\} \}
$$
be the \emph{tangential polar arcs}.  

To a polar arc $\gamma$ one may associate the positive rationals $h$ and $\delta$ defined  as follows:
\begin{equation}\label{eq:deltaandh}
 h:=\ord f(\gamma(y),y), \quad \delta:= \max_{\zeta \in Z(f)}\{\ord (\gamma(y)-\zeta(y))\}.
\end{equation}

In terms of the Newton polygon relative to $\gamma$ (cf \cite{kuo-par}) denoted by $\PP(f, \gamma)$, these numbers have the following interpretation:  $(0,h)$ is  the point where $\PP(f, \gamma)$ intersects the vertical axis and $\delta$ is the co-slope of the highest edge of $\PP(f,\gamma)$.

To a polar arc $\gamma$ one  may  also associate the bar $B(\gamma)$ in the Kuo-Lu tree on which $\gamma$ grows \cite{kuo-lu}.

\begin{example}
It is possible that two polars arc have the same $\delta$ but grow on different bars. For instance, let $f(x,y)=\left[(x-y^2)^2-y^6 \right](x-y^3)(x-y^4)$. We have four Puiseux roots
$$
\zeta_1(y)=y^2-y^3,\quad \zeta_2(y)=y^2+y^3,\quad \zeta_3(y)=y^3,\quad \zeta_4(y)=y^4
$$
and three polar arcs:
$$
\gamma_1(y)=\frac{1}{2}y^3+\hot, \quad \gamma_2(y)=y^2-y^4+\hot, \quad \gamma_3(y)=\frac{1}{2}y^2+\hot.
$$
We get the equalities: $\delta(\gamma_1)=\delta(\gamma_2)=3$, $h(\gamma_1)=h(\gamma_2)=10$,  $\PP(f, \gamma_1)=\PP(f,\gamma_2)$. The curves ${\gamma_1}_{*}$, and ${\gamma_2}_{*}$ have the same  tangent  line $\{x=0\}$. By examining the Kuo-Lu tree, we see that, despite the fact that $\gamma_1$ and $\gamma_2$ have the same numbers $k,\delta$ and $h$, they grow on different bars (Figure \ref{tree1}).
\begin{figure}[H]
                \centering
\begin{tikzpicture}
\draw (0.5,1)--(3.5,1);
\draw (0.5,1)--(0.5,2);
\draw (0,2)--(1,2);
\draw (0,2)--(0,3);
\draw (1,2)--(1,3);
\draw (3.5,1)--(3.5,2);
\draw (3,2)--(4,2);
\draw (3,2)--(3,3);
\draw (4,2)--(4,3);
\draw (2,0)--(2,1);

\draw[dashed] (0.5,2.)-- (0.5,2.5);
\draw[dashed] (2,1.)-- (2,1.5);
\draw[dashed] (3.5,2.)-- (3.5,2.5);

\begin{scriptsize}
\draw (0.2,1) node {$B_1$};
\draw (3.7,1) node {$2$};
\draw (-0.3,2) node {$B_2$};
\draw (1.2,2) node {$3$};
\draw (2.7,2) node {$B_3$};
\draw (4.2,2) node {$3$};

\draw (0.,3.2) node {$\zeta_1$};
\draw (1.,3.2) node {$\zeta_2$};

\draw (3.,3.2) node {$\zeta_3$};
\draw (4.,3.2) node {$\zeta_4$};

\draw (0.5,2.7) node {$\gamma_2$};
\draw (2.,1.7) node {$\gamma_3$};
\draw (3.5,2.7) node {$\gamma_1$};

\end{scriptsize}
\end{tikzpicture}

\caption{The Kuo-Lu tree of $f$}\label{tree1}
\end{figure}

\end{example}

\smallskip

We introduce more notations. 

\begin{definition}\label{def_pol_clust}
The following subsets  of $P(f)$ are called  \emph{polar clusters}:
$$
\PC_{\delta,h,B}=\{\gamma \in P(f) \mid  \max_{\zeta \in Z(f)}\{\ord (\gamma(y)-\zeta(y))\}=\delta, \ord f(\gamma(y),y)=h, B(\gamma)=B\}\subset P(f),
$$
$$\PC_{k,\delta,h,B} =\{ \gamma \in \PC_{\delta,h,B} \mid \gamma_* \mbox{ is tangent to } L_k \} \subset P(f).$$
\end{definition}


If $(k, \delta,h,B) \not= (k', \delta',h',B')$ then the polar clusters  $\PC_{k, \delta,h,B}$ and $\PC_{k', \delta',h',B'}$ are disjoint by definition.  One may compare the polar clusters $\PC_{k,\delta,h,B}$ to the ``paquets'' defined by Garcia Barroso in terms of the Eggers tree, see \cite{GB1, GB2}.

\medskip


Let $g:(\bC^2,0)\to(\bC,0)$ be another  holomorphic function, and let $\cone_0 \{g=0\} = \cup_{k=1}^{r'}L'_{k}$ be its tangent cone. By definition  $f$ and $g$ are \emph{topologically equivalent} if  $g = f\circ \varphi$, where  $\varphi$ is a homeomorphism.
In this case we have $r=r'$, and the topological equivalence induces a bijective correspondence $L_{k}\mapsto  L'_{k}$ (see  \cite{Pa}).

\begin{theorem}\label{polar_cluster2}
Let $f,g:(\bC^2,0)\to(\bC,0)$ be holomorphic functions such that  $f\stackrel{\top}{\sim}g$. The topological equivalence induces a bijective correspondence: 
\begin{equation}\label{eq:topbij}
\bigsqcup_{k,\delta,h,B}\PC_{k, \delta,h,B}(f) \to \bigsqcup_{k,\delta,h,B}\PC_{k, \delta,h,B}(g)
\end{equation}
 which maps $\PC_{k,\delta,h,B}(f)$ to $\PC_{k,\delta,h,B}(g)$.
\end{theorem}

The proof is based on the following reformulation of \cite[Lemma 3.3]{kuo-lu}, where the first statement is just the first part of \cite[Lemma 3.3]{kuo-lu}. It was remarked in \cite[Remark 3.2]{GLP} that  the second part of \cite[Lemma 3.3]{kuo-lu} is not true. However, from the proof of  \cite[Lemma 3.3]{kuo-lu} one can extract the correct statement given below (second claim).  See also \cite[Theorem 3.1]{vietnam} for an explicit proof of it. 

\begin{lemma}\label{KLvietnam} 
For $\zeta_1,\zeta_2\in Z(f)$ there exists a polar arc $\gamma$ of $f$ such that:
\begin{equation}\label{klequat}
\begin{split}
&\max_{\zeta \in Z(f)}\{\ord (\gamma(y)-\zeta(y))\}=\ord(\gamma(y)-\zeta_1(y))\\
&=\ord(\gamma(y)-\zeta_2(y))=\ord(\zeta_1(y)-\zeta_2(y))
\end{split}
\end{equation}
 and $\zeta_1$, $\zeta_2$, $\gamma$ grow on the same bar in the Kuo-Lu tree of $f$.
Moreover, for each fixed  polar arc $\gamma \in P(f)$, there exist two Puiseux roots $\zeta_1,\zeta_2$ of $f$  satisfying the equalities \eqref{klequat}  and $\zeta_1$, $\zeta_2$, $\gamma$ grow on the same bar in the Kuo-Lu tree of $f$.
\fin
\end{lemma}

The proof of Theorem \ref{polar_cluster2} uses several known facts, as listed below.

\smallskip
\noindent\emph{Fact 1.} By \cite{Pa},  the topological equivalence of $f$ and $g$ defines a bijective correspondence between the irreducible components of the zero sets $f^{-1}(0)$ and $g^{-1}(0)$ that preserves the multiplicities of these components and the contact orders of any pairs of distinct components.  Consequently, this bijective correspondence preserves the multiplicities of the Puiseux roots and the set of contacts $\{\ord (\zeta_i(y)-\zeta_j(y)) \mid \zeta_i, \zeta_j \in Z(f), j\neq i\}$,  together with the bars on which the Puiseux roots grow. Moreover, if $\zeta_1,\zeta_2$ are Puiseux roots of $f$ such that $\zeta_{1,*}$ and $\zeta_{2,*}$ have a common tangent, then this correspondence maps $\zeta_i$ to $\zeta'_i \in Z(g)$, $i=1,2$,  such that $\zeta'_{1,*}$ and $\zeta'_{2,*}$ have a common tangent.

\medskip

\noindent\emph{Fact 2.} Let  $f$ be mini-regular in $x$ of order $m$,  and let $f_m=c\prod_{i=1}^r(a_ix-b_iy)^{r_i}$ be the initial form of $f$, where $c, a_i \in \bC^*$, $b_{i}\in \bC$. We claim that there exists a tangential polar arc $\gamma\in P^T(f)$  such that $\gamma(y)=ay +\hot$, $a\in \bC$, if and only if  
there is $j \in \{1, \dots, r\}$ such that $\frac{b_j}{a_j}=a$ and $r_j \geq 2$.

 This is obvious if $a=0$. When $a\neq 0$ (equivalently,  $b_j/a_j \neq 0$), then one applies Lemma \ref{KLvietnam}  observing that $\gamma$ is a tangential polar arc if and only if there exists a Puiseux roots $\zeta\in Z(f)$ such that $\ord \gamma(y)<\ord(\gamma(y)-\zeta(y))$.

\medskip 
\begin{proof}[Proof of Theorem \ref{polar_cluster2}]
Without loss of generality, one may assume that $f$ is mini-regular in $x$ of order $m$.
Let $\gamma \in \PC_{k,\delta, h,B}(f)$. By Lemma \ref{KLvietnam} and Fact 2, there exist Puiseux roots $\zeta_1, \zeta_2\in Z(f)$ such that $\zeta_{1,*}$ and $\zeta_{2,*}$ are tangent to $L_k$ and \eqref{klequat} is satisfied.  Additionally, $\zeta_1$ and $\zeta_2$ grow on $B$. Therefore:\begin{align*}
h=\ord f(\gamma(y),y)&=\sum_{i=1}^p m_i\ord(\gamma(y)-\zeta_i(y))\\
&=\sum_{i=1}^p m_i\ord(J^{\delta}(\gamma)(y)-\zeta_i(y))\\
&=\sum_{i=1}^p m_i \min\{\ord (\zeta_1(y)-\zeta_i(y)),\ord (\zeta_2(y)-\zeta_i(y))\},
\end{align*} 
where $m_i$ is multiplicity of a Puiseux root $\zeta_i$ and $p$ is the number of distinct Puiseux roots of $f$.
By the bijective correspondence we have $\zeta_i \mapsto \zeta'_i$ with the same contact. By Lemma \ref{KLvietnam}, for each pair of Puiseux roots $\zeta'_i,\zeta'_j \in Z(g)$ there exists a polar arc $\gamma'\in P(g)$ such that \eqref{klequat} is satisfied. Moreover, $\gamma'$ grows on the same bar as $\zeta'_i,\zeta'_j$. By Fact 1, Fact 2 and the definition \eqref{eq:deltaandh}, we get $\gamma'\in\PC_{k,\delta,h,B}(g)$. 
\end{proof}

\begin{remark}\label{nontangent}
Let $m$ be the multiplicity of $f$, which is  a topological invariant   by Burau \cite{Bur} and Zariski \cite{Zar}. The polar cluster $\PC_{1,m}(f)$ (i.e. $\delta=1$ and $h=m$) is actually the set of the non-tangential polar arcs of $f$. By \cite{KKP}, the number of polar arcs in $\PC_{1,m}(f)$ is equal to $r-1$, where $r$ is the number of the distinct lines in $\cone_{0} \{f=0\}$. By Theorem \ref{polar_cluster2}, if $f\stackrel{\top}{\sim}g$, then the bijective correspondence \eqref{eq:topbij}  maps $\PC_{1,m}(f)$ to $\PC_{1,m}(g)$, and the number of polar arcs in $\PC_{1,m}(f)$ and in $\PC_{1,m}(g)$ is the same.
\end{remark}

\medskip

One may define two other types of polar clusters:
$$
\PC_{\delta}(f)=\{\gamma\in P(f) \mid \max_{\zeta \in Z(f)}\{\ord (\gamma(y)-\zeta(y))\}=\delta\}, 
$$
$$
\PC_{h}(f)=\{\gamma \in P(f) \mid \ord f(\gamma(y),y)=h\}.
$$

By Theorem \ref{polar_cluster2} and Remark \ref{nontangent}, we then have:

\begin{corollary}\label{special_cluster}
Let $f,g:(\bC^2,0)\to(\bC,0)$ be holomorphic functions such that $f\stackrel{\top}{\sim}g$. Then the topological equivalence induces bijective correspondences  
$$
\bigsqcup_{\delta}\PC_{\delta}(f) \to \bigsqcup_{\delta}\PC_{\delta}(g), \mbox{ and } \quad \bigsqcup_{h}\PC_{h}(f) \to \bigsqcup_{h}\PC_{h}(g)
$$
which map $\PC_{\delta}(f)$ to $\PC_{\delta}(g)$,  and $\PC_{h}(f)$ to $\PC_{h}(g)$,  respectively.
\fin
\end{corollary}
\medskip

\begin{remark}
 If $f:(\bC^2,0)\to(\bC,0)$ is a \emph{Newton non-degenerate} holomorphic function, then  $\delta=\delta' $ if and only if $h=h'$ if and only if $B=B'$.
 
 If $f\stackrel{\top}{\sim}g$, and if $f$ and $g$ are Newton non-degenerate and both have isolated singularity\footnote{``Isolated singularity'' is implied for instance by the condition \emph{nearly convenient}, see \cite[Property 3.2]{Len}.}, then the bijective correspondence of clusters $\PC_{\delta}(f)$ $\mapsto$ $\PC_{\delta}(g)$ also preserves the number of polar arcs in each cluster. One can show the following equality:
$$\#\PC_{\delta} (f) = \left\{ \begin{array}{ll}
\#\{\zeta \in Z(f) \mid \ord \zeta(y)=\delta\} & \textrm{ if } \delta\neq \max_{\zeta \in Z(f)}\{\ord \zeta(y)\}\\
\#\{\zeta \in Z(f) \mid \ord \zeta(y)=\delta\}-1 & \textrm{ if } \delta = \max_{\zeta \in Z(f)}\{\ord \zeta(y)\}
\end{array} \right..$$
\end{remark}

\subsection{The \L ojasiewicz exponents of polar clusters}\label{ss:lojcluster}

For any $k,\delta$ we define the \emph{partial \L ojasiewicz exponent of $\PC_{k,\delta}$} as follows: 
$$
\rho_{k,\delta}(f):=(\max_{\gamma\in \PC_{k,\delta}(f)}\{\ord f(\gamma(y),y)\} -1) / \max_{\gamma\in \PC_{k,\delta}(f)}\{\ord f(\gamma(y),y)\}.$$

From Theorem \ref{polar_cluster2} we immediately derive :

\begin{corollary}\label{c:lojacluster}
The partial \L ojasiewicz exponents $\rho_{k,\delta}$ are topological invariants. 
\fin
\end{corollary}

This topological invariance of partial \L ojasiewicz exponents is more refined  than the well-known topological invariance of the \L ojasiewicz exponent $\rho_{0}(f)$, cf \cite{Te, kuo-lu, Pa, vietnam}, since we have
$\rho_{0}(f) = \max \{\rho_{k,\delta}(f)\}_{k,\delta}$.

\medskip

\subsection{The tangential \L ojasiewicz exponents}\label{ss:lojtang}


 Let $f:(\bC^2,0)\rightarrow(\bC,0)$ be a mini-regular holomorphic function, without any assumption about singularities.
As before, $\cone_0 \{f=0\} = \cup_{k=1}^{r}L_{k}$ is the decomposition of the tangent cone into distinct  tangent lines intersecting at $0$. 

Let $\Gamma^{\tan}_{k}(l, f)$ be the subset of polar branches of $f$ which are tangent to  $L_{k}$.  
The polar curves depend on  the generic  $l$,  and the subset $\Gamma^{\tan}_{k}(l, f)$ too.  Let   $Q_{k}(f)\subseteq Q(f)$  the set of polar quotients of all the polars in $\Gamma^{\tan}_{k}(l,f)$, and let us recall that $0\notin Q(f)$. It follows from  \cite{Te} that  $Q_{k}(f)$ does not depend on the generic $l$.

The  \L ojasiewicz exponent of all the non-tangential polar arcs is equal to $(m-1) /m$, where $m$ is the multiplicity of $f$.
This is equal to $\min  Q(f)$ and it is a topological invariant, as well as the multiplicity $m$. 
In the general case it follows from Theorem \ref{teissier}(d), see also \cite{vietnam},  that $\min Q(f)$ is a topological invariant.

 From Theorem \ref{polar_cluster2} we may derive:


\begin{corollary}\label{c:lojacluster2}
 The  tangential \L ojasiewicz exponent
$$\rho_{0k} :=(\max Q_{k}(f) -1) / \max Q_{k}(f)$$ 
 is a topological invariant, for any $k\in 1, \dots , r$.  
 \fin
  \end{corollary}

\medskip
 \subsection{The tangential Milnor numbers}\label{ss:milnortang}
 
We have seen that $ \mu(f) = \sum_{j} m_{j}(q_{j}-1)$, where $j$ runs over all the polar branches $\Gamma_{j}(l,f)\subset \Gamma(l,f)$.
By the above result we may define the \emph{tangential Milnor numbers} as:
\[ \mu_{k}(f) := \sum_{j} m_{j}(q_{j}-1),
\] 
where $j$ runs over the set $\Gamma^{\tan}_{k}(l, f)$ 
 of polars branches  which are tangent to $L_{k}$. 
It then follows also from Theorem \ref{polar_cluster2}:
\begin{corollary}\label{c:lojacluster3}
 $\mu_{k}(f)$ is a topological invariant, for any $k\in \{1, \ldots , r\}$.
 \fin
\end{corollary}

\bigskip

 Let us point out that the above results give no information about the multiplicities of the polar quotients, i.e. the number of polar arcs which  have the same polar quotient.


\subsection{Examples}
\begin{example}[Polar arcs with the same $h$ and different $\delta$]\label{newex1}
Let us consider $f(x,y)=(x-y^2)^2(x-y^3)^2(x-y^4)(x-2y^4)(x-3y^4)-6y^{22}$. We have seven Puiseux roots:
$$
\zeta_1(y)=4y^4+\hot, \quad \zeta_{2,3}(y)=(1\pm i\sqrt{2})y^4+\hot,
$$
$$
\zeta_{4,5}(y)=y^3 \pm \sqrt{6}y^{\frac{9}{2}}+\hot, \quad \zeta_{6,7}(y)=y^2 \pm \sqrt{6}y^{6}+\hot
$$
and six polar arcs:
$$
\gamma_1(y)=y^2, \quad \gamma_2(y)=\frac{5}{7}y^2+\hot, \quad \gamma_3(y)=y^3,
$$
$$
\gamma_4(y)=\frac{3}{5}y^3+\hot, \quad \gamma_{5,6}(y)=\frac{6\pm \sqrt{3}}{3}y^4+\hot.
$$
Obviously:
$$
\ord f(\gamma_1(y),y)=\ord f(\gamma_3(y),y)=\ord f(\gamma_{5,6}(y),y)=22,
$$
$$
\ord f(\gamma_2(y),y)=14, \quad \ord f(\gamma_4(y),y)=19
$$
and ${\gamma_{i}}_{*}$ are tangent to $\{x=0\}$, $i=1,\dots, 6$ but
$$
\delta(\gamma_1)=6, \quad \delta(\gamma_2)=2, \quad \delta(\gamma_3)=\frac{9}{2},\quad \delta(\gamma_4)=3,
\quad
\delta(\gamma_{5,6})=4.
$$
\begin{figure}[H]
                \centering
\begin{tikzpicture}
\draw (4.5,0.)-- (4.5,1.);
\draw (2.5,1.)-- (5.5,1.0);
\draw (2.5,1.0)-- (2.5,2.);
\draw (1.0,2.)-- (3.5,2.0);
\draw (1.,2.)-- (1.,3.);
\draw (3.5,2.)-- (3.5,3.5);
\draw (0.,3.)-- (2.,3.);
\draw (0.,3.)-- (0.,4.);
\draw (0.,3.)-- (0.,4.);
\draw (1.,3.)-- (1.,4.);
\draw (2.,3.)-- (2.,4.);
\draw (3.,3.5)-- (4,3.5);
\draw (3.,3.5)-- (3.,4.5);
\draw (4.,3.5)-- (4., 4.5);
\draw (5.5,1.)-- (5.5,4.5);
\draw (5.,4.5)-- (6.,4.5);
\draw (5.,4.5)-- (5.,5.5);
\draw (6.,4.5)-- (6., 5.5);

\draw[dashed] (0.5,3.)-- (0.5,3.5);
\draw[dashed] (1.5,3.)-- (1.5,3.5);
\draw[dashed] (2.5,2.)-- (2.5,2.5);
\draw[dashed] (3.5,3.5)-- (3.5,4.);
\draw[dashed] (4.5,1.)-- (4.5,2.5);
\draw[dashed] (5.5,4.5)-- (5.5,5.);

\begin{scriptsize}
\draw (5.5,5.2) node {$\gamma_1$};
\draw (5.,5.7) node {$\zeta_6$};
\draw (6.,5.7) node {$\zeta_7$};
\draw (4.5,2.7) node {$\gamma_2$};
\draw (3.5,4.2) node {$\gamma_3$};
\draw (2.5,2.7) node {$\gamma_4$};
\draw (1.5,3.7) node {$\gamma_6$};
\draw (0.5,3.7) node {$\gamma_5$};
\draw (0.,4.2) node {$\zeta_1$};
\draw (1.,4.2) node {$\zeta_2$};
\draw (2.,4.2) node {$\zeta_3$};
\draw (3.,4.7) node {$\zeta_4$};
\draw (4.,4.7) node {$\zeta_5$};

\draw (2.2,1) node {$B_1$};
\draw (5.7,1) node {$2$};

\draw (0.7,2) node {$B_2$};
\draw (3.7,2) node {$3$};

\draw (-0.3,3) node {$B_3$};
\draw (2.2,3) node {$4$};

\draw (2.7,3.5) node {$B_4$};
\draw (4.3,3.5) node {$\frac{9}{2}$};

\draw (4.7,4.5) node {$B_5$};
\draw (6.2,4.5) node {$6$};

\end{scriptsize}
\end{tikzpicture}

\caption{The Kuo-Lu tree of $f$}\label{pic_ex1}
\end{figure}

From the above Kuo-Lu tree of $f$ we identify five polar clusters:
$$
\PC_{k,6,22,B_5}(f)=\{\gamma_1\}, \quad \PC_{k,2,14, B_1}(f)=\{\gamma_2\}, \quad \PC_{k,9/2,22,B_4}(f)=\{\gamma_3\},
$$
$$
\PC_{k,3,19,B_2}(f)=\{\gamma_4\}, \quad \PC_{k,4,22,B_3}(f)=\{\gamma_5,\gamma_6\},
$$
 and also the following five polar clusters of  type $(k,\delta)$:
 $$
\PC_{k,6}(f)=\{\gamma_1\}, \quad \PC_{k,2}(f)=\{\gamma_2\}, \quad \PC_{k,9/2}(f)=\{\gamma_3\},
$$
$$
\PC_{k,3}(f)=\{\gamma_4\}, \quad \PC_{k,4}(f)=\{\gamma_5,\gamma_6\},
$$
where $L_k=\{x=0\}$.  We also have:
$$
Q_{k,6}(f)=\{22\}, \quad  Q_{k,2}(f)=\{14\}, \quad Q_{k,9/2}(f)=\{22\}
 $$
 $$
Q_{k,3}(f)=\{19\}, \quad  Q_{k,4}(f)=\{22\},
 $$
where $Q_{k,\delta}(f):= \{\ord f(\gamma(y),y)\mid \gamma\in \PC_{k,\delta}\}$.

The \L ojasiewicz exponents of the polar clusters of a type $(k,\delta)$  are (cf. \S \ref{ss:lojcluster}):
$$
\varrho_{k,6}(f)=\frac{\max Q_{k,6}(f) -1}{\max Q_{k,6}(f)}=\frac{21}{22}, \quad \varrho_{k,2}(f)=\frac{\max Q_{k,2}(f) -1}{\max Q_{k,2}(f)}=\frac{13}{14}, 
$$
$$
\varrho_{k,9/2}(f)=\frac{\max Q_{k,9/2}(f) -1}{\max Q_{k,9/2}(f)}=\frac{21}{22},\quad \varrho_{k,3}(f)=\frac{\max Q_{k,3}(f) -1}{\max Q_{k,3}(f)}=\frac{18}{19},
$$
$$
\quad \varrho_{k,4}(f)=\frac{\max Q_{k,4} (f)-1}{\max Q_{k,4}(f)}=\frac{21}{22}.
$$

One has $Q_k(f)=\{14,19,22\}$. The tangential \L ojasiewicz exponent is  (cf. \S \ref{ss:lojtang}):
$$
\varrho_{0k}(f)=\frac{\max Q_k(f)-1}{\max Q_k(f)}=\frac{21}{22}.
$$ 

Let us point out that the polar curve $\Gamma(y,f)$ has six irreducible components $\gamma_{i,*},\dots, \gamma_{6,*}$, all  of multiplicity $m_i=1$. Therefore, the tangential Milnor number is (cf. \S \ref{ss:milnortang}):
$$
\mu_k(f)=\sum_{i=1}^6[\ord f(\gamma_i(y),y)-1]=115.
$$ 

\end{example}
 
\begin{example}[Polar arcs with the same $\delta$ and different $h$]\label{newex2}
 Let $f(x,y)=[(x-y^2)^2-y^{10}][(x-y^3)^2-y^{10}][(x-y^4)^2-y^{10}](x^2-y^{10})$. We have here eight Puiseux roots:
$$
\zeta_1(y)=y^2-y^{5},\quad \zeta_2(y)=y^2+y^{5},\quad \zeta_3(y)=y^3-y^{5},\quad \zeta_4(y)=y^3+y^{5},
$$
$$
\zeta_5(y)=y^4-y^{5},\quad \zeta_6(y)=y^4+y^{5},\quad \zeta_7(y)=y^5,\quad \zeta_8(y)=-y^5
$$
and seven polar arcs:
 $$
 \gamma_1(y)=-y^6+ \hot, \quad  \gamma_2(y)=\frac{1}{2}y^4+ \hot, \quad \gamma_3(y)=\frac{2}{3}y^3+ \hot ,\quad \gamma_4(y)=\frac{3}{4}y^2+\hot,
 $$
 $$
 \gamma_5(y)=y^4+y^6+\hot,\quad \gamma_6(y)=y^3+2y^7+\hot, \quad \gamma_7(y)=y^2+3y^8+\hot.
 $$
We moreover have:
 $$
 \delta(\gamma_1)= \delta(\gamma_5)= \delta(\gamma_6)= \delta(\gamma_7)=5,
 $$
 $$
  \delta(\gamma_2)=4, \quad  \delta(\gamma_3)=3, \quad  \delta(\gamma_4)=2
 $$
and also:
$$
\ord f(\gamma_1(y),y)=\ord f(\gamma_5(y),y)=28,\quad \ord f(\gamma_2(y),y)=\ord f(\gamma_6(y),y)=26,
$$
$$
\ord f(\gamma_4(y),y)=16, \quad \ord f(\gamma_3(y),y)=\ord f(\gamma_7(y),y)=22.
$$
The Kuo-Lu tree of $f$ is given in Figure \ref{pic_ex2}.
\begin{figure}[H]
                \centering
\begin{tikzpicture}
\draw (0.5,1.)-- (3.5,1.);
\draw (1.5,0.)-- (1.5,1);
\draw (0.5,1)-- (0.5,4);
\draw (0,4)-- (1,4);
\draw (0,4)-- (0,5);
\draw (1,4)-- (1,5);
\draw (3.5,1)-- (3.5,2);
\draw (2.5,2)--(5.5,2);
\draw (2.5,2)--(2.5,4);
\draw (2,4)-- (3,4);
\draw (2,4)-- (2,5);
\draw (3,4)-- (3,5);
\draw (5.5,2)--(5.5,3);
\draw (4.5,3)--(6.5,3);
\draw (4.5,3)--(4.5,4);
\draw (6.5,3)--(6.5,4);
\draw (4,4)-- (4,5);
\draw (4,4)--(5,4);
\draw (5,4)--(5,5);
\draw (6,4)--(7,4);
\draw (6,4)--(6,5);
\draw (7,4)--(7,5);

\draw[dashed] (0.5,4)--(0.5,4.5);
\draw[dashed] (1.5,1)--(1.5,3);
\draw[dashed] (2.5,4)--(2.5,4.5);
\draw[dashed] (3.5,2)--(3.5,2.5);
\draw[dashed] (4.5,4)--(4.5,4.5);
\draw[dashed] (5.5,3)--(5.5,3.5);
\draw[dashed] (6.5,4)--(6.5,4.5);

\begin{scriptsize}
\draw (0,5.2) node {$\zeta_1$};
\draw (0.5,4.7) node {$\gamma_7$};
\draw (1,5.2) node {$\zeta_2$};
\draw (1.5,3.2) node {$\gamma_4$};
\draw (2,5.2) node {$\zeta_3$};
\draw (2.5,4.7) node {$\gamma_6$};
\draw (3,5.2) node {$\zeta_4$};
\draw (3.5,2.7) node {$\gamma_3$};
\draw (4,5.2) node {$\zeta_5$};
\draw (4.5,4.7) node {$\gamma_5$};
\draw (5,5.2) node {$\zeta_6$};
\draw (5.5,3.7) node {$\gamma_2$};
\draw (6,5.2) node {$\zeta_7$};
\draw (6.5,4.7) node {$\gamma_1$};
\draw (7,5.2) node {$\zeta_8$};

\draw (-0.3,4) node {$B_4$};
\draw (1.7,4) node {$B_5$};
\draw (3.7,4) node {$B_6$};
\draw (5.7,4) node {$B_7$};
\draw (7.2,4) node {$5$};
\draw (0.2,1) node {$B_1$};
\draw (3.7,1) node {$2$};
\draw (2.2,2) node {$B_2$};
\draw (5.7,2) node {$3$};
\draw (4.2,3) node {$B_3$};
\draw (6.7,3) node {$4$};

\end{scriptsize}
\end{tikzpicture}

\caption{The Kuo-Lu tree of $f$}\label{pic_ex2}
\end{figure}

We identify the following seven polar cluster:
$$
\PC_{k,5,28,B_7}=\{\gamma_1\},\quad \PC_{k,4,26,B_3}=\{\gamma_2\},\quad \PC_{k,3,22,B_2}=\{\gamma_3\},\quad \PC_{k,2,16,B_1}=\{\gamma_4\},
$$
$$
\PC_{k,5,28,B_6}=\{\gamma_5\},\quad \PC_{k,5,26,B_5}=\{\gamma_6\},\quad \PC_{k,5,22,B_4}=\{\gamma_7\}.
$$

Let us compute the \L ojasiewicz exponent of the polar clusters (cf. \S \ref{ss:lojcluster}). We have $\PC_{k,5}=\{\gamma_1,\gamma_5,\gamma_6,\gamma_7\}$, $\PC_{k,4}=\{\gamma_2\}$, $\PC_{k,3}=\{\gamma_3\}$ and $\PC_{k,2}=\{\gamma_4\}$, where $L_k=\{x=0\}$,
 and:

 $$
 Q_{k,5}(f)=\{22,26,28\}, \quad  Q_{k,4}(f)=\{26\},\quad Q_{k,3}(f)=\{22\}, \quad  Q_{k,2}(f)=\{16\}.
 $$
where $Q_{k,\delta}(f):= \{\ord f(\gamma(y),y)\mid \gamma\in \PC_{k,\delta}\}$. 
Therefore:
 $$
 \varrho_{k,5}(f)=\frac{\max Q_{k,5}(f) -1}{\max Q_{k,5}(f)}=\frac{27}{28},\quad \varrho_{k,4}(f)=\frac{\max Q_{k,4}(f) -1}{\max Q_{k,4}(f)}=\frac{25}{26},
$$
 $$
  \varrho_{k,3}(f)=\frac{\max Q_{k,3}(f) -1}{\max Q_{k,3}(f)}=\frac{21}{22},\quad \varrho_{k,2}(f)=\frac{\max Q_{k,2}(f) -1}{\max Q_{k,2}(f)}=\frac{15}{16}.
 $$
 
Let us note that $Q_k(f)=\{16,22,26,28\}$. We  get the tangential \L ojasiewicz exponent (cf. \S \ref{ss:lojtang}):
$$
\varrho_{0k}(f)=\frac{\max Q_{k}(f) -1}{\max Q_{k}(f)}=\frac{27}{28}.
$$

The tangential Milnor number (cf. \S \ref{ss:milnortang}) is:
$$
\mu_k(f)=\sum_{i=1}^7[\ord f(\gamma_i(y),y)-1]=161,
$$
since the polar curve $\Gamma(y,f)$ has seven irreducible components $\gamma_{i,*}$, all seven of multiplicity $m_i=1$.

\end{example}

\section{The bi-Lipschitz correspondence}

Let $f = g \circ \varphi$. If  $\varphi$ is bi-Lipschitz, it appears that there is a bijective  correspondence between a more refined version of clusters, based on the gradient canyons. We follow  \cite{PT} where these ideas were developed.
\begin{definition}\label{d:canyon} \cite{KKP} 
Let $\gamma$ be a polar arc of $f$ such that $f(\gamma(y),y)\not \equiv 0$. The \textit{gradient degree} $d_{\gr}(\gamma)$ is the smallest number $q$ such that 
$$
\ord_y(\|\Grad f(\gamma(y),y)\|)=\ord_y(\|\Grad f(\gamma(y)+uy^q,y)\|), 
$$
holds for generic $u\in \bC$. The \emph{gradient canyon} $\GC(\gamma_{*})$ is the subset of all the curve germs $\alpha_{*}$, where $\alpha$ is a Puiseux arc of the form
$$\alpha(y)  := \gamma(y)+uy^{d_{\gr}(\gamma)} + \hot$$ 
for any $u\in \bC$.  We may and will call $d_{\gr}(\gamma_*)$   the \emph{canyon degree}. When $\gamma$ is a polar arc such that $f(\gamma(y),y)\equiv 0$ then we set $d_{\gr}(\gamma_{*})=+\infty$.
\end{definition}

Let  $\cC(f)$ denote the set of all gradient canyons of $f$ of degree $>1$, and let $\cC_{k}(f)\subseteq \cC(f)$ be the subset of canyons which are tangent to the line $L_{k}$. One has a disjoint union $\cC(f) = \bigsqcup_{k=1}^{r}\cC_{k}(f)$.

For some fixed $\alpha_*$, where $\alpha$ is a Puiseux series as before, one defines the infinitesimal disks: 
$$
\mcd^{(e)}(\alpha_*;\rho):=\{\beta_*\mid \beta(y)=[J^{e}(\alpha)(y)+cy^e]+\hot,|c|\leq \rho\},
$$
where $1\leq e<\infty$, $\rho\geq 0$.

Consider $\mcd^{(e)}(\alpha_*;\rho)$ of finite order $e\geq 1$ and finite radius $\rho>0$, and a compact  ball $B(0;\eta):=\{(x,y)\in \bC^2 \mid \sqrt{|x|^2+|y|^2}\leq \eta\}$ with small enough $\eta >0$.  
Let then
$$
\Horn^{(e)}(\alpha_*;\rho;\eta):=\{ (x,y)\in B(0;\eta)\mid x=\beta(y)=J^e(\alpha)(y)+cy^e, |c|\leq \rho\}
$$
be the \emph{horn domain} associated to $\mcd^{(e)}(\alpha_*;\rho)$.

Let $\mcd^{(e')}_{\gamma_*,\varepsilon}(\lambda;\eta)$ be the union of disks in the Milnor fibre $\{f=\lambda\}\cap B(0;\eta)$ of $f$ defined as follows 
$$
 \mcd^{(e')}_{\gamma_*,\varepsilon}(\lambda;\eta):=\{f=\lambda\} \cap \Horn^{(e')}(\gamma_*;\varepsilon;\eta),
$$
for some rational $e'$ close enough to $d_{\gamma}$, with $1<e'<d_{\gamma}$, for small enough $\varepsilon >0$, and where by ``disk'' we mean ``homeomorphic to an open disk''. 

We still refer to canyon disks of canyon degree $>1$.  Let $D_{f}$ be some disk cut out on the Milnor fibre  $f^{-1}(\lambda)$ by some horn $\Horn^{(d)}(\gamma_{g,*};\e;\eta)$ of a canyon $\GC(\gamma_{*})$ of degree  $d =\deg D_{f}$.   
The radius of $D_{f}$ is  $k |y|^{d}\sim_{\ord}|\lambda|^{d/h}$, for some $k>0$ and that  the distance between two conjugated disks is of order $\ord | y|$.  
If two polars are in the same canyon, then their associated disks coincide (by definition).

By ``canyon disk''  (cf. \cite[\S 5]{PT}) we shall mean in the following such a disk of radius order $d/h$ with respect to $|\lambda|$, modulo some multiplicative constant $>0$ which is not specified. We then have the following:


\begin{theorem}\label{t:main0}\cite{PT}
If the map $\varphi$, is  bi-holomorphic then the image by  $\varphi$ of a canyon is a canyon, and the degree is preserved.

If the map $\varphi$ is bi-Lipschitz, then it induces a bijection between the canyon disks of $f$ and the canyon disks of $g$ by preserving the canyon degree.
 \fin
\end{theorem}

This key theorem yields bi-Lipschitz invariants, as follows.

\noindent 
Let $f = g \circ \varphi$, for a bi-Lipschitz homeomorphism  $\varphi$.  Theorem \ref{t:main0}  implies that certain subsets of gradient canyons,  defined in terms of  the mutual contact,  are preserved by $\varphi$. 
Moreover, there are well-defined ``clusters'' of canyons of $f$ which are sent by $\varphi$ into similar clusters of $g$,
and that such clusters are determined by certain rational numbers which are  bi-Lipschitz invariants.

It is shown in  \cite[\S 5.3]{PT} that the ``contact orders of canyons'' is well-defined, and that  yields a more refined partition of each subset $\cC_{k}$.
We consider the tangential canyons only, i.e. those of degree more than 1.   
As we have seen before in case of topological equivalence, the tangent cone $\cone_0 \{ f=0\}$ is a topological invariant, and the contact orders between the Puiseux roots of $f$  are  preserved by $\varphi$.
More interesting phenomena appear after the polar arcs leave the tree, namely  at a higher level than the co-slope $\delta = \tan \theta_{B(h)}$.
Whenever the canyons $\GC(\gamma_{1*})\ni \alpha_{*}$ and $\GC(\gamma_{2*})\ni \beta_{*}$ are different and both of degree $d>1$, their contact  order is lower than $d$ and 
therefore it does not depend  on the choice of $\alpha_{*}$ in the first canyon, and of $\beta_{*}$ in the second canyon.
This yields a well-defined \emph{contact order between two distinct canyons of degree $d$}. Their contact can be greater or equal to the co-slope of the corresponding bar, say $\tan \theta_{B(h)}$, but less than their canyon degrees.


\subsection{Clustering the canyons}
 We may cluster the canyons in terms of the essentials bars of the tree of $f$, 
  namely those canyons departing from the essential bars  $B(h)$, i.e. bars for which there are   polar arcs leaving the tree of $f$ on them. 
  The contact order of distinct canyons, can be greater or equal to the co-slope  $\tan \theta_{B(h)}$ of the corresponding bar,  but less than their canyon degrees.

  Let  $\cC_{k, d, B(h)}(f)\subset \cC_{k}$ be the union of canyons of degree $d>1$,  which  grow on the same bar $B(h)$,  and thus have the same $h$,  where $d>\tan \theta_{B(h)} >1$, more precisely those canyons of degree $d$ with the same top edge of the Newton polygon relative to any polar arc of the canyon.  

One then has the disjoint union decomposition:
$$
 \cC_{k}(f) = \bigsqcup_{d>1, h}\cC_{k, d, B(h)}(f). 
$$
Note that each canyon from $\cC_{k, d, B(h)}(f)$ has the same contact $>1$ with a fixed irreducible component  $\{f_{i} =0\}$.


 \medskip 
 
Next,  each cluster   $\cC_{k, d, B(h)}(f)$ has a partition into unions of canyons according to the mutual order of contact between canyons.  More precisely, a fixed  gradient canyon  $\GC_i(f) \subset  \cC_{k, d, B(h)}(f)$ has a well defined order of contact $k(i,j)$ with some other gradient canyon $\GC_j(f) \subset  \cC_{k, d, B(h)}(f)$ from the same cluster; we count also the \emph{multiplicity} of each such contact, i.e. the number of  canyons $\GC_j(f)$ from the cluster $\cC_{k, d, B(h)}(f)$ which have exactly the same contact with $\GC_i(f)$.

Let then $K_{k, d,B(h), i}(f)$ be the (un-ordered) set of those contact orders $k(i,j)$ of the fixed canyon $\GC_i(f)$, counted with multiplicity, and varying index $j$.

 Let now $\cC_{k, d, B(h), \omega}(f)$ be the union of canyons from $\cC_{k, d,B(h)}(f)$ which have exactly the same set $\omega = K_{k, d,B(h), i}(f)$ of  orders of contact with the other canyons from $\cC_{k, d,B(h)}(f)$.  This defines a partition:
 
$$
\cC_{k, d, B(h)}(f) = \bigsqcup_{\omega} \cC_{k, d,B(h), \omega}(f).
$$
 
In this way each canyon of $\cC_{k}$ has its ``identity card'' composed  of these  contact orders (which are rational numbers),  and it belongs to a certain  cluster  $\cC_{k, d, B(h), \omega}(f)$ in the partition of $\cC_{k}(f)$. 
 It is possible that two canyons have the same ``identity card''.  
 We clearly have, by definition, the inclusions:
 \[  \cC_{k}(f) \supset \cC_{k,d, B(h)}(f) \supset \cC_{k, d, B(h), \omega}(f)
 \]
 for any defined indices.

With these notations, one has:

\begin{theorem}\label{t:main}\cite{PT}
The bi-Lipschitz map $\varphi$ induces a bijection between the gradient canyons of $f$ and those of $g$.
For any $k$, any degree $d>1$, any bar $B(h)$  and any rational  $h$, the following are  bi-Lipschitz invariants:
\begin{enumerate}
\rm \item \it   the cluster of canyons $\cC_{k, d, B(h)}(f)$.
\rm \item \it the set of contact orders $K_{k, d,B(h), i}(f)$,  and for each such set, the sub-cluster of canyons $\cC_{k, d,B(h), K_{k, d,B(h), i}}(f)$.
\end{enumerate}
\sloppy
Moreover, $\varphi$  preserves the contact orders between any two clusters of type
 $\cC_{k, d,B(h), K_{k, d,B(h), i}}(f)$.
\end{theorem}

\medskip

\subsection{The HP-invariant and gradient canyons, after \cite{HP} and \cite{PT}}\label{s:hp-invar} \ \\
  We will explain in terms of canyons how the bi-Lipschitz invariants found by Henry and Parusinski  \cite{HP}  occur.
   Let us  fix the variable $y$
and express any polar arc of $f$ as an expansion $(\gamma(y),y)$. 

We then have:
\[ f(\gamma(y),y) = a y^{h} + \hot \]

 By \cite{PT}, for any polar arc $\gamma_{f}$ of $f$ there is some polar arc $\gamma_{g}$ of $g$
such that $\varphi(\gamma_{f}(y),y) \in \GC(\gamma_{g*})$ of canyon degree $d= d_{\gamma_{g}}= d_{\gamma_{f}}$.
We thus have:
\begin{equation}\label{eq:gamma}
   \| \varphi(\gamma_{f}(y),y) - (\gamma_{g}(Y), Y)\| \sim | Y^{d}|, 
\end{equation}

where $ |Y| \sim | \varphi_{2}(\gamma_{f}(y), y)| \sim |y|$.

Let us consider another polar arc $\gamma'_{f}$ such that it has contact $>1$ with $\gamma_{f}$.
In two variables, this implies that these two polar arcs are tangent to one of the (singular) lines, call it $L$, in the tangent cone of $\{f=0\}$.

 We can write for $\gamma'_{f}$ a relation similar to \eqref{eq:gamma} but the local variable $Y$ is not the same; in principle there is  another well-defined local variable $Y'$ with $|Y'| \sim | \varphi_{2}(\gamma'_{f}(y), y)| \sim |y|$ such that:
\[  \| \varphi(\gamma'_{f}(y),y) - (\gamma'_{g}(Y'), Y')\| \sim | Y'^{d'}|  .
\]
Since we have $\ord_{0}\| \varphi(\gamma_{f}(y),y) -\varphi(\gamma'_{f}(y),y)\|  >1$
we deduce that the fraction  $Y/Y'$  tends to $1$, and that the fractions $Y/y$ and $Y'/y$ tend to the same constant $c \not= 0$ whenever one of them converges.

A bi-Lipschitz map $\varphi$ sends a canyon of $f$ to a canyon of  $g$ of the same degree.
Consider the set $A_{L} = $ all canyons of $f$ tangent to some line $L\in \cone_0\{f=0\}$.  
A  canyon $\cC$ of $A_{L}$ yields a  couple  $(d_{\cC}, a_{h_{\cC}})\in \bQ_{+}\times \bC$, where :
 \[  d_{\cC}:=\mathrm{degree\  of\  the\  canyon} \]
 \[  a_{h_{\cC}}:= \mathrm{ coefficient\ of\ } y^{h} \ \mathrm{\  in\ the\ expansion\ of\ } f(\gamma(y),y),  \]
 for some polar arc $\gamma$ in $\cC$ (and where $a_{h_{\cC}}$ is independent on the polar arcs in $\cC$, and even on the arcs of the canyon, as we have seen before). 
 
 By the above computation we have:
\begin{theorem} \rm (\cite{HP}, \cite{HP2} supplemented by \cite{PT})\ \it  \\
 The effect of the bi-Lipschitz map $\varphi$ on each such couple $(d_{\cC}, a_{h_{\cC}})$ is:  the identity on $d_{\cC}$,  and the multiplication  of $a_{h_{\cC}}$
by  $c^{h_{\gamma}}$,   where $\gamma$ is some polar arc in $\cC$ and where $c$  is a certain non-zero constant which is the same for all canyons $\cC$ in $A_{L}$.
\fin
\end{theorem}


\bigskip

 



\begin{thebibliography}{MMMM}



\bibitem[BR]{bori} J. Bochnak, J.-J. Risler, \emph{Sur les exposants de {\L}ojasiewicz} Comment. Math. Helv. 50 (1975), 493-507.

\bibitem[BK]{BK} E. Brieskorn, H. Kn\"orrer, \emph{Plane algebraic curves}, Basel-Boston-Stuttgart: Birkh\"auser Verlag. VI, 721 p.(1986).
\bibitem[Brz]{Brz} S. Brzostowski, 
\emph{The {\L}ojasiewicz Exponent of Semiquasihomogeneous Singularities}, Bull. London Math. Soc. 47 (2015) 848-852.

\bibitem[Bur]{Bur} W. Burau, \emph{Kennzeichung der Schlauchknoten}, Abh. Math. Sem. Hamburg, 9 (1932), 125-133.


\bibitem[GB1]{GB1}E. R. Garci\'a Barroso, \emph{Un th\' eor\`eme de d\' ecomposition pour les polaires g\' en\' eriques d'une courbe plane}. C. R. Acad. Sci. Paris S\' er. I Math. 326 (1998), no. 1, 59-62.

\bibitem[GB2]{GB2}E. R. Garci\'a Barroso,  \emph{Sur les courbes polaires d'une courbe plane r\' eduite}.  Proc. London Math. Soc. (3) 81 (2000), no. 1, 1-28.

\bibitem[GLP]{GLP} J. Gwo\'zdziewicz, A. Lenarcik, A. P{\l}oski, \emph{Polar invariants of plane curve singularities: intersection theoretical approach}. Demonstr. Math. 43 (2010), no. 2, 303-323.



\bibitem[HP1]{HP} J-P. Henry, A. Parusi\'nski,  \emph{Existence of moduli for bi-Lipschitz equivalence of analytic functions.} Compositio Math. 136 (2003), no. 2, 217-235.

\bibitem[HP2]{HP2} J-P. Henry, A. Parusi\'nski, \emph{Invariants of bi-Lipschitz equivalence of real analytic functions}. in: Geometric singularity theory, 67-75, Banach Center Publ., 65, Polish Acad. Sci., Warsaw, 2004.

\bibitem[HNP]{vietnam}
P.D. Hoang, H.D. Nguyen, T.S. Pham, 
\emph{Topological invariants of plane curve singularities: polar quotients and Lojasiewicz gradient exponents},  Internat. J. Math. 30 (2019), no. 14, 1950073, 19 pp. 

\bibitem[KOP]{KOP} T. Krasi\'nski, G. Oleksik and A. P{\l}oski, 
\emph{The {\L}ojasiewicz exponent of an isolated weighted homogeneous surface singularity}, Proc. Amer. Math. Soc. 137 (2009), 3387-3397.

\bibitem[KL]{kuo-lu} T.-C. Kuo,  Y.C. Lu,
\emph{On analytic function germs of two complex variables}, Topology 16 (1977), 299-310.

\bibitem[KuPar]{kuo-par} T.-C. Kuo,  A. Parusi\'nski,
\emph{Newton polygon relative to an arc}, in: Real and Complex Singularities (Sao Carlos, 1998), Chapman and Hall Res. Notes Math. 412 (2000), 76-93.

\bibitem[KKP]{KKP} 
S. Koike, T.-C. Kuo, L. P\u aunescu, \emph{A'Campo curvature bumps and the Dirac phenomenon near a singular point},
Proc. Lond. Math. Soc. (3) 111 (2015), no. 3, 717-748.





\bibitem[LT]{LT} M. Lejeune-Jalabert, B. Teissier, \emph{Cl\^oture int\'egrale des id\'eaux et \'equisingularit\'e}, Centre Math\'ematiques, Universit\'e Scientifique et Medical de Grenoble (1974).  Ann. Fac. Sci. Toulouse Math. (6) 17 (2008),  no. 4, 781-859.

\bibitem[Len]{Len} A. Lenarcik, \emph{On the {\L}ojasiewicz exponent of the gradient of a holomorphic function}, Banach Center Publ. 44 (1998), 149-166.


\bibitem[Lo]{Lo} S. {\L}ojasiewicz, \emph{Ensembles semi-analytiques}, preprint IHES, 1965.

\bibitem[Me]{Me} M. Merle, \emph{Invariants polaires des courbes planes}. Invent. Math. 41 (1977), no. 2, 103-111.

\bibitem[Pa]{Pa}
A.  Parusi\'nski, \emph{A criterion for topological equivalence of two variable
complex analytic function germs}, Proc. Japan Acad. 84, Ser. A (2008), 147-150.


\bibitem[PT]{PT} 
L. P\u aunescu, M. Tib\u ar, \emph{Concentration of curvature and Lipschitz invariants of holomorphic functions of two variables},  J. London Math. Soc. 100 (2019) no.1, 203-222.


\bibitem[Te1]{Te1} B. Teissier, \emph{Introduction to equisingularity problems}, Proceedings
of Symposia in Pure Mathematics, Algebraic Geometry, Arcata 1974, Vol. 29.

\bibitem[Te2]{Te}
B. Teissier, \emph{Vari\' et\' es polaires  I. Invariants polaires des singularit\' es d'hypersurfaces}
Inventiones math. 40  (1977),  267-292.
\bibitem[Wa]{Wa} R. J. Walker, \emph{Algebraic Curves}, Dover, 1962.
\bibitem[Za1]{Zar} O. Zariski, \emph{On the Topology of Algebroid Singularities}, Amer. J. Math. 54 (1932), no. 3, 453-465.
\bibitem[Za2]{Zar2} O. Zariski, \emph{Studies in equisingularity. I: Equivalent singularities of plane algebroid curves.} Am. J. Math. 87, 507-536 (1965). 




\end{thebibliography}
\end{document}